\documentclass[12pt,reqno]{amsart}
\usepackage{amssymb,amsmath,amsthm,mathrsfs,verbatim}
\usepackage[all, knot]{xy}
\usepackage{rotating}
\usepackage{cancel}
\usepackage[normalem]{ulem}
\usepackage[OT2,T1]{fontenc}
\newtheorem{theorem}{Theorem}
\newtheorem{lemma}[theorem]{Lemma}
\newtheorem{corollary}[theorem]{Corollary}

\newtheorem{proposition}[theorem]{Proposition}

\renewcommand{\AA}{\mathcal{A}}

\theoremstyle{remark}

\newtheorem*{remarks}{Remarks}

\numberwithin{theorem}{section} \numberwithin{equation}{section}
\numberwithin{figure}{section}

\newcommand{\cM}{\mathcal{M}}

\newcommand{\wt}{\kappa}

\newcommand{\HH}{\mathbb{H}}

\newcommand{\QQ}{\mathcal{Q}}

\newcommand{\ChiQ}{\chi}
\newcommand{\CoH}{\mathcal{H}}

\newcommand{\R}{\mathbb{R}}
\newcommand{\C}{\mathbb{C}}
\newcommand{\CC}{\mathcal{C}}

\newcommand{\Z}{\mathbb{Z}}
\newcommand{\N}{\mathbb{N}}

\newcommand{\SL}{{\text {\rm SL}}}
\newcommand{\PSL}{{\text {\rm PSL}}}

\newcommand{\shin}{\mathscr{S}^*}

\newcommand{\re}{\textnormal{Re}}

\def\H{\mathbb{H}}
\def\ca{{\mathfrak a}}
\def\cb{{\mathfrak b}}
\newcommand{\Hcusp}{H_{2-2k}^{\text{cusp}}}

\renewcommand{\SS}{\mathbb{S}}

\newcommand{\DD}{\mathcal{D}}

\begin{document}
\title{On cycle integrals of weakly holomorphic modular forms}

\author{Kathrin Bringmann} 
\address{Mathematical Institute\\University of
Cologne\\ Weyertal 86-90 \\ 50931 Cologne \\Germany}
\email{kbringma@math.uni-koeln.de}
\author{Pavel Guerzhoy}
\address{Department of Mathematics\\ University of Hawaii\\ Honolulu, HI 96822-2273}
\email{pavel@math.hawaii.edu}
\author{Ben Kane}
\address{
Department of Mathematics\\ University of Hong Kong\\ Pokfulam\\ Hong Kong}
\email{bkane@maths.hku.hk}
\date{\today}
\thanks{The research of the first author was supported by the Alfried Krupp Prize for Young University Teachers of the Krupp Foundation and by the Deutsche Forschungsgemeinschaft (DFG) Grant No. BR 4082/3-1. The research of the second author was partially supported by a Simons Foundation Collaboration grant (209914 to P. Guerzhoy).  Part of this research was conducted while the third author was a postdoc at the University of Cologne}
\subjclass[2010] {11F37, 11F11}
\keywords{Cycle integrals,  weakly holomorphic modular forms, harmonic weak Maass forms, Shintani lift, Shimura lift}

\begin{abstract}
In this paper, we investigate cycle integrals of weakly holomorphic modular forms.  We show that these integrals coincide with the cycle integrals of classical cusp forms.  We use these results to define a Shintani lift from integral weight weakly holomorphic modular forms to half-integral weight holomorphic modular forms.
\end{abstract}

\maketitle

\section{Introduction and statement of results} 
The relationship between integral and half-integral weight modular forms has been pivotal in a number of applications.  The link between these spaces was first established in a groundbreaking paper of Shimura \cite{Shimura} in which he constructed lifts from half-integral to integral weight modular forms.  The level of the image of Shimura's lifts was then studied by a number of authors.  Niwa \cite{Niwa} used theta lifts to find the optimal level of the lifts in general, while Kohnen \cite{KohnenMathAnn} found a distinguished subspace of half-integral weight modular forms (known as Kohnen's plus space) which map to a lower level.  A number of general results have a much more explicit form when restricted to Kohnen's plus space.  For example, Waldspurger \cite{Waldspurger} used Shimura's lifts to establish connections between central values of $L$-functions of integral weight modular forms and coefficients of corresponding half-integral weight modular forms.  By restricting to Kohnen's plus space, Kohnen and Zagier \cite{KohnenZagier} refined Waldspurger's formula \cite{Waldspurger} to determine an explicit constant which proved non-negativity of the central values of these $L$-functions.  Waldspurger's formula \cite{Waldspurger} (and Kohnen and Zagier's variant \cite{KohnenZagier}) also yields a link between the vanishing of central $L$-values of integral weight modular forms and coefficients of half-integral weight modular forms.  Specifically, for each fundamental discriminant $D$ and integral weight Hecke eigenform $f$, the central value of the $D$th twist of the $L$-function of $f$ vanishes if and only if the $|D|$th coefficient of a half-integral weight modular form related to $f$ by the Shimura correspondence is zero.

More recently, Bruinier and Ono \cite{BruinierOno} proved that the central values of derivatives of $L$-functions vanish precisely when certain coefficients of a newer modular object, known as harmonic weak Maass forms, are algebraic.  

Central to the theory built around integral and half-integral weight modular forms is a two-sided interplay between the Shimura lifts and lifts of Shintani \cite{Shintani} which map integral weight to half-integral weight modular forms.  The Shintani lifts are defined in terms of certain cycle integrals, which are, roughly speaking, integrals of modular objects along geodesics defined by certain integral binary quadratic forms of positive discriminant.  In particular, for a continuous function $F$ satisfying weight $2\wt\in 2\Z$ modularity, the coefficients of the Shintani lifts are given as traces of cycle integrals of the type 
\begin{equation}\label{eqn:cycdef}
\CC\left(F;Q\right):=\int_{C_Q}F(z)Q(z,1)^{\wt-1} dz,
\end{equation}
where $Q$ is an integral binary quadratic form of positive discriminant and $C_Q$ is explicitly defined in Section \ref{sec:setup}.  In \cite{BGK}, the authors 
related two families of 
cycle integrals coming from 
weight $2-2k$
 ($k\in \N$ fixed throughout) harmonic weak Maass forms $\cM$.  
One family is formed with 
the traces of cycle integrals defining the $\delta$th Shintani lift of the weight $2k$ cusp form $\xi_{2-2k}(\cM)$  (where $\xi_{2-2k}:=2iy^{2-2k}\overline{\frac{\partial}{\partial \overline{z}}}$ and $\delta$ is a fundamental discriminant for which $(-1)^k\delta>0$).  These cycle integrals were shown to be 
fundamentally equal to the traces of cycle integrals of a (non-holomorphic) weight $0$ Maass form which is essentially built out of the $(k-1)$th (holomorphic) derivative of $\cM$.   However, if one instead takes $2k-1$ derivatives of $\cM$, then the resulting function is a weight $2k$ \begin{it}weakly holomorphic modular form\end{it} (i.e., a meromorphic modular form whose only possible poles occur at the cusps).  Cycle integrals of weakly holomorphic modular forms are well-defined and were first considered in \cite{DITrational}, where an exciting connection between such cycle integrals and rational period functions was established.  Since $\xi_{2-2k}(\cM)$ and $\DD^{2k-1}(\cM)$ (where $\DD:=\frac{1}{2\pi i} \frac{\partial}{\partial z}$) are both weight $2k$ weakly holomorphic modular forms, it is natural to compare their cycle integrals.  Defining a regularization, which is necessary to guarantee covergence of \eqref{eqn:cycdef} in the case that the discriminant of $Q$ is a square, our first result is that these integrals in essence coincide for every element $\cM$ in the space $\Hcusp$ of harmonic weak Maass forms for which $\xi_{2-2k}(\cM)$ is a cusp form.  
\begin{theorem}\label{thm:cycleequal}
For every harmonic weak Maass form $\cM\in\Hcusp$, every discriminant $D>0$, and every $Q\in\QQ_D$, the cycle integrals satisfy 
\begin{equation}\label{eqn:cycleequal}
\CC\left(\xi_{2-2k}(\cM);Q\right)=-\frac{\left(4\pi\right)^{2k-1}}{(2k-2)!}\overline{\CC\left(\DD^{2k-1}(\cM);Q\right)}.
\end{equation}
\end{theorem}
Note that the kernel of $\xi_{2-2k}$ (namely, the space $M_{2-2k}^!$ of weakly holomorphic modular forms) is non-trivial.  In the special case that $\cM\in M_{2-2k}^!$, Theorem \ref{thm:cycleequal} therefore immediately yields the following corollary.
\begin{corollary}\label{cor:weakly}
If $G\in M_{2-2k}^!$, then
$$
\CC\left(\DD^{2k-1}(G);Q\right)=0.
$$
\end{corollary}

By taking a trace of the cycle integrals, we next define a Shintani lift for $f\in S_{2k}^!$, the space of \begin{it}weakly holomorphic cusp forms\end{it} (i.e., those $f\in M_{2k}^!$ whose constant terms vanish).  For $\delta$ given before Theorem \ref{thm:cycleequal} and $f\in S_{2k}^!$, this Shintani lift is given by 
\begin{equation}\label{eqn:shindef}
\shin_{\delta}(f)(z):= \sum_{\substack{m>0\\ (-1)^km\equiv 0,1\pmod{4}}}  \sum_{Q\in \QQ_{|\delta|m}/\SL_2(\Z) }\ChiQ_{\delta}(Q) \CC\left(f;Q\right) q^m,
\end{equation}
where the genus character $\ChiQ_{\delta}$ is defined in \eqref{eqn:omegadef} and $q:=e^{2\pi iz}$.  We obtain a relation between \eqref{eqn:shindef} and the classical $\delta$th Shintani lift.   If $f\in S_{2k}$, then a theorem of Shintani \cite{Shintani} implies that $\shin_{\delta}(f)$ is an element of 
Kohnen's plus space
 \cite{KohnenMathAnn} $\SS_{k+\frac{1}{2}}$ of weight $k+\frac{1}{2}$ cusp forms.  

We prove that this extends to $S_{2k}^!$ and include the relation between the weight $2k$ Eisenstein series $G_{2k}$ (see \eqref{eqn:Eisdef}) and Cohen's weight $k+\frac{1}{2}$ Eisenstein series $\CoH_k$ (defined in \eqref{eqn:Cohendef}) for completeness.
\begin{theorem}\label{thm:rmodular}
Suppose that $\delta$ is a fundamental discriminant satisfying $(-1)^k\delta>0$.
\noindent

\noindent
\begin{enumerate}
\item
If $f\in S_{2k}^!$, then $\shin_{\delta}(f)\in \SS_{k+\frac{1}{2}}$.
\item
If $\cM\in \Hcusp$, then 
\begin{equation}\label{eqn:shinrel}
\shin_{\delta}\left(\xi_{2-2k}(\cM)\right)=-\frac{\left(4\pi\right)^{2k-1}}{(2k-2)!}\left(\shin_{\delta}\left(\DD^{2k-1}(\cM)\right)\right)^c,
\end{equation}
where for $g\in \SS_{k+\frac{1}{2}}$ we use $g^c(z):=\overline{g (-\overline{z})}$ for any function $g$ on the upper half-plane. 

\item
If $\delta=1$, then we have 
$$
\shin_1\left(G_{2k}(z)\right)= \frac{1}{2}\zeta\left(1-k\right)\CoH_k(z).
$$
\end{enumerate}
\end{theorem}
\begin{remarks}
\noindent

\noindent
\begin{enumerate}
\item If $\delta=1$, then \eqref{eqn:shinrel} holds for all harmonic weak Maass forms.
\item We expect that a twisted version of Theorem \ref{thm:rmodular} (3) holds for general 
fundamental discriminants $\delta$ satisfying $(-1)^k \delta > 0$.
\item One can prove that for $\delta\neq 1$  and $f\in S_{2k}^!$, the Shintani lift $\shin_{\delta}(f)$ is not influenced by the even periods of $f$. That is to say, if $f_1, f_2\in S_{2k}^!$ have the same odd periods, then $\shin_{\delta}\left(f_1\right)=\shin_{\delta}\left(f_2\right)$.  Computational evidence further indicates that in this case $\CC\left(f_1;Q\right)=\CC\left(f_2;Q\right)$ for every $Q\in \QQ_D$ and every $D>0$.  Since we do not need this for our purposes, we omit the proof here.
\item 
In \cite{FM} another geometric extension of the Shintani lift was given, which for Eisenstein series coincides with our definition \eqref{eqn:shindef}. Theorem \ref{thm:rmodular} (3) can also directly be concluded from Section 9 of \cite{FM}.
\item 
The restriction that $\delta$ is fundamental may be omitted by using the action of the Hecke operators to define the Shintani lift for non-fundamental discriminants, but we do not work out the details here.
\end{enumerate}
\end{remarks}

The paper is organized as follows.  In Section \ref{sec:setup}, we introduce the regularization necessary to define the cycle integrals for square discriminants and relate the cycle integrals for non-square discriminants to the periods of weakly holomorphic modular forms.  In Section \ref{sec:mainproofs} we prove Theorem \ref{thm:cycleequal}.  Finally, we prove Theorem \ref{thm:rmodular} in Section \ref{sec:Shintani}.

\section*{Acknowledgements}
The authors thank Jens Funke for helpful discussion.

\section{Cycle integrals and periods}\label{sec:setup}

In this section, we recall known facts on cycle integrals and periods of weakly holomorphic modular forms and additionally prove a relation between them.  The generalization of these definitions to weakly holomorphic modular forms requires a regularized integral defined by the first author, Fricke, and Kent \cite{BFK}.  Assume that $f:\HH\to \C$ is a continuous function and that there exists $c\in \R^+$ such that $f$ satisfies the growth condition 
\begin{equation}\label{bound}
f(z)=O\Big(e^{c \, \mathrm{Im}(z)}\Big)
\end{equation}
uniformly in $\mathrm{Re}(z)$ as  $\text{Im}(z)\to\infty$. Then, for each $z_0 \in \mathbb{H}$, the integral
\[
	\int_{z_0}^{i\infty} e^{uiw} f(w) \; dw
\]
(where the path of integration lies within a  vertical strip) is convergent
for $u \in \C$ with $\mathrm{Re}(u) \gg 0$.  If this integral has an analytic continuation to $u=0$, then the {\it regularized integral} of $f$ is given by
$$
R.\int_{z_0}^{i\infty}f(w) \; dw:=\left[\int_{z_0}^{i\infty}e^{uiw}f(w) \;dw\right]_{u=0}\text{,}
$$
where the right-hand side means that we take the value at $u=0$ of the analytic continuation of the integral. We similarly define integrals at other cusps $\mathfrak a$.  Specifically, suppose that $\mathfrak a=\sigma_{\mathfrak a}(i \infty)$ for a scaling matrix $\sigma_{\mathfrak a} \in\SL_2(\Z)$.  If $f(\sigma_{\mathfrak a} z)$ satisfies \eqref{bound}, then we use the regularization
\[
R.\int_{z_0}^{\mathfrak a}f(w) \; dw:=R.\int_{\sigma_{\mathfrak a}^{-1} z_0}^{i\infty}f \big|_2\sigma_a(w) \; dw.
\]
For cusps $\ca, \cb$, we set
\begin{equation}\label{eqn:regdef}
R.\int_{\ca}^{\cb}f(w) \; dw:=R.\int_{z_0}^{\cb}f(w) \; dw + R.\int_{\ca}^{z_0}f(w) \; dw
\end{equation}
for any $z_0 \in \mathbb H.$  It is not hard to see that this integral is independent of $z_0 \in \mathbb H.$

This regularization can be used to generalize the classical definition of periods of cusp forms.  To be more precise, the \begin{it}$n$th periods\end{it} ($n\in \N_0$) for \begin{it}weakly holomorphic cusp forms\end{it} $f\in S_{2k}^!$ (those weakly holomorphic modular forms whose constant terms vanish) are given by
$$
r_n(f):=R.\int_{0}^{\infty}f(it)t^ndt,
$$
which may be packaged into \begin{it}period polynomials\end{it}
$$
r(f;z):=\sum_{n=0}^{2k-2}i^{-n+1}\binom{2k-2}{n}r_n(f) z^{k-2-n}.
$$
The first two authors, Kent, and Ono (see (1.7) and (1.9) of \cite{BGKO}) proved that the period polynomials of $f(z)=\sum_{n\in\Z} a_n q^n \in S_{2k}^!$ are constant multiples of the error to modularity under $S:=\left(\begin{smallmatrix} 0 &-1\\ 1&0\end{smallmatrix}\right)$ of the \begin{it}(holomorphic) Eichler integrals\end{it} 
$$
\mathcal{E}_f(z):=\sum_{n\in \Z\setminus\{0\}}\frac{a_n}{n^{2k-1}} q^n.
$$
For $t_0>0$, $c\in\Z$, and $d\in \N$, we also require the twisted $L$-series
$$
L_f^\ast \left( \zeta_c^d, s \right):=  \sum_{m\geq m_0} \frac{a_f(m) \zeta_c^{dm}\Gamma (s, 2\pi mt_0)}{(2\pi m)^{s}} + i^k c^{k-2s} \sum_{m\geq m_0} \frac{a_f(m) \zeta_c^{-am}\Gamma \left(k-s, \frac{2\pi m}{c^2 t_0}\right)}{(2\pi m)^{k-s}}.
$$
Here for $y>0$, $\Gamma\left(s,y\right):=\int_{y}^{\infty} t^{s-1} e^{-t}dt$ is the \begin{it}incomplete $\Gamma$-function\end{it}, and the definition can be shown to be independent of the choice of $t_0$.  

Before clarifying the role played by the above regularization if the discriminant is a square, we formally define the cycle integrals given in \eqref{eqn:cycdef}.  For each non-square discriminant $
D
>0$, denote the set of integral binary quadratic forms of discriminant $
D
$ by $\QQ_{
D
}$.  If $
D
$ is not a square, then the (infinite cyclic) group of automorphs of an indefinite binary quadratic form $Q=[a,b,c]\in \QQ_{
D
}$
 is defined by
$$
g_Q:=\left(\begin{matrix}\frac{t+bu}{2}& cu\\ -au& \frac{t-bu}{2}\end{matrix}\right),
$$
where $(t,u)$ is the smallest positive solution to the Pell equation $t^2- 
D
 u^2=4$ (the automorphs are trivial if $
D
$ is a square, and we let $g_Q$ be the identity in this case).  For $a\neq 0$ (resp. $a=0$), let $S_Q$ be the oriented semicircle (resp. vertical line) given by 
$$
a\left|z\right|^2 + bx+c=0
$$
directed counterclockwise if $a>0$, clockwise if $a<0$, and up from the real axis if $a=0$.  Note that $\frac{dz}{Q(z,1)}$ is an invariant measure on $S_Q$.  Since  every $Q\in \QQ_{
D
}$ is modular of weight $-2$ for $\Gamma_Q:=\langle g_Q\rangle$, for $\wt\in \Z$ it is natural to integrate the product of $Q(z,1)^{\wt-1}$ times functions $F$ satisfying weight $2\wt\in 2\Z$ 
modularity
on $\Gamma_Q$ along $\Gamma_Q\backslash S_Q$.  Hence, for $z\in S_Q$, we define
$C_Q$ to be the directed arc from $z$ to $g_Qz$ along $S_Q$ if $
D
$ is non-square and $C_Q:=S_Q$ if $
D
$ is a square.  As in Lemma 6 of \cite{DIT}, one concludes that the cycle integral in \eqref{eqn:cycdef} is well-defined for any continuous function $F$ satisfying weight $2k\in \Z$ modularity on $\Gamma_Q$ for which the integral converges.  

More generally, in the case of convergence, we extend the definition of the cycle integrals in \eqref{eqn:cycdef} by
$$
\CC\left(F;Q\right):=R.\int_{C_Q}F(z)Q(z,1)^{k-1} dz.
$$
In particular, this is necessary for square $
D
$ whenever $F$ does not decay towards the cusps.  Hence this regularization allows one to take cycle integrals of \begin{it}harmonic weak Maass forms.\end{it}  Denoting the weight $2-2k$ \textit{hyperbolic Laplacian} by
\begin{equation*}
\Delta_{2-2k} := -y^2\left( \frac{\partial^2}{\partial x^2}+\frac{\partial^2}{\partial y^2}\right) + i(2-2k) y\left(\frac{\partial}{\partial x}+i \frac{\partial}{\partial y}\right),
\end{equation*}
these are functions $\cM:\H\to\C$ satisfying:
\begin{enumerate}
\item[(i)] $\cM\big|_{2-2k} \gamma = \cM$ for all $\gamma \in \SL_2(\Z)$,
\item[(ii)] $\Delta_{2-2k}\left(\cM\right)=0$,
\item[(iii)] $\cM$ has at most linear exponential growth at $i\infty$.
\end{enumerate}
The subspace of harmonic weak Maass forms $\cM$ for which $\xi_{2-2k}(\cM)$ is a cusp form is denoted by $\Hcusp$.  We aim to compare the cycle integrals of the weakly holomorphic modular forms $\CC\left(\xi_{2-2k}\left(\cM\right);Q\right)$ and $\overline{\CC\left(\DD^{2k-1}\left(\cM\right);Q\right)}$ for $\cM\in \Hcusp$.

Having established the definition of cycle integrals for weakly holomorphic modular forms, we have almost all of the notation required for the definition of the Shintani lift \eqref{eqn:shindef}.  To complete the definition, recall that for every pair of discriminants $D_1$ and $D_2$, the corresponding \begin{it}genus character\end{it} (for example, see pp. 59--62 of \cite{Siegel}) of a binary quadratic form $Q\left(X,Y\right)=\left[a,b,c\right]\left(X,Y\right):=aX^2+bXY+cY^2\in \QQ_{D_1D_2}$ is given by
\begin{equation}\label{eqn:omegadef}
\ChiQ_{D_1}\left(Q\right):=
\begin{cases}
\left(\frac{D_1}{r}\right) & \text{if }\left(a,b,c, D_1\right)=1\text{ and $Q$ represents $r$ with $\left(r,D_1\right)=1$,}\\ 0 & \text{if }\left(a,b,c,D_1\right)>1.
\end{cases}
\end{equation}
Here $\left(\frac{D_1}{\cdot}\right)$ denotes the Kronecker character.  

We make frequent use of a generalization of Theorem 7 of \cite{KohnenZagierRational} to include weakly holomorphic modular forms.  To state the theorem, we require some notation.  A binary quadratic form $Q=[a,b,c]\in \QQ_{D}$ ($D>0$) is called \begin{it}reduced\end{it} if $a>0$, $c>0$, and $b>a+c$.  
Let $\AA$ be an $\text{SL}_2(\Z)$-equivalence class of integral binary quadratic forms of discriminant $D>0$.  We enumerate the reduced forms in $\AA$ by $Q_0,\dots, Q_r=Q_0$, where 
\begin{equation}\label{eqn:reduced}
Q_j=Q_{j-1}\circ M_j:=Q_{j-1}\big|_{-2}M_j,
\end{equation}
with 
\begin{equation}\label{eqn:Mjdef}
M_j:=\left(\begin{matrix} m_j&1\\ -1&0\end{matrix}\right)
\end{equation}
for some integers $m_j\geq 2$.  Here $|_{-2}$ is the usual weight $-2$ slash operator and we use $Q_{j-1}|_{-2}$ to abbreviate the action on the function $Q_{j-1}(z,1)$.  We use the reduced forms to define the polynomial
$$
Q_{k,D,\mathcal{A}}(X):=\sum_{\substack{Q\in \AA\\ Q\text{ reduced}}} Q(X,-1)^{k-1} = \sum_{j=1}^r Q_j\left(X,-1\right)^{k-1}
$$
and denote the coefficient of $X^n$ by $q_{k,D,\AA}^{(n)}$.  We are now ready to give the relation between cycle integrals and periods, which was shown for cusp forms in \cite{KohnenZagierRational}.
\begin{theorem}\label{thm:cycleperiod}
For each equivalence class $\AA$ of binary quadratic forms of non-square discriminant containing $Q\in \AA$ and $f\in S_{2k}^!$, one has
\begin{equation}\label{eqn:cycleperiod}
\CC\left(f;Q\right)=\sum_{n=0}^{2k-2}i^{-n+1}q_{k,D,\AA}^{(n)} r_n(f).
\end{equation}
\end{theorem}
\begin{proof}
We rewrite the right-hand side of (\ref{eqn:cycleperiod}), using the definition of $r_n$
\begin{multline}\label{eqn:rnregThm7}
\sum_{n=0}^{2k-2}i^{-n+1}q_{k,D,\AA}^{(n)} r_n(f)=\sum_{n=0}^{2k-2}(-1)^n q_{k,D,\AA}^{(n)}R.\int_0^{i\infty}f(z)z^n dz\\
= R.\int_0^{i\infty}f(z)Q_{k,D,\AA}(-z)dz=R.\sum_{j=1}^{r}\int_0^{i\infty}f(z) Q_{j-1}(z,1)^{k-1}dz.
\end{multline}

On the left-hand side of \eqref{eqn:cycleperiod}, we follow the formal argument at the beginning of the proof in \cite{KohnenZagierRational} mutatis mutandis to obtain 
$$
\CC\left(f;Q\right)=\sum_{j=1}^{r}\int_{z_0}^{M_jz_0}f(z)Q_{j-1}(z,1)^{k-1}dz.
$$
However, for every $z_0\in \HH$, continuity in $u$ implies that
\begin{equation}\label{eqn:regThm7}
\left[\sum_{j=1}^{r}\int_{z_0}^{M_jz_0}f(z)e^{uiw}Q_{j-1}(z,1)^{k-1}dz\right]_{u=0}=\sum_{j=1}^{r}\int_{z_0}^{M_jz_0}f(z)Q_{j-1}(z,1)^{k-1}dz.
\end{equation}
Since $\CC\left(f;Q\right)$ is independent of $z_0$, the left-hand side of \eqref{eqn:regThm7} is hence also independent.  We then take $z_0\to 0$, precisely yielding the right-hand side of \eqref{eqn:rnregThm7}.
\end{proof}

\section{Proof of Theorem \ref{thm:cycleequal}}\label{sec:mainproofs}
The argument to prove Theorem \ref{thm:cycleequal} naturally separates for $D$ square and non-square.  We begin with the case that $D$ is not a square and first prove the existence of $\cM\in \Hcusp$ for which \eqref{eqn:cycleequal} holds.  

\begin{proposition}\label{prop:cycleequalnonsq}
For every $g\in S_{2k}$, there exists $\cM\in \Hcusp$ for which $\xi_{2-2k}(\cM)=g$ such that for every non-square discriminant $D$ and every $Q\in \QQ_D$, \eqref{eqn:cycleequal} holds.
\end{proposition}
\begin{proof}
We may choose $\cM$ as in the second statement of Theorem 3.1 of \cite{BFK}, namely satisfying 
$$
r\left(\xi_{2-2k}\left(\cM\right);z\right)=-\frac{\left(4\pi\right)^{2k-1}}{(2k-2)!}\overline{r\left(\DD^{2k-1}\left(\cM\right);z\right)}.
$$
By comparing coefficients of $z^n$, one obtains
$$
r_n\left(\xi_{2-2k}\left(\cM\right)\right)=(-1)^n\frac{\left(4\pi\right)^{2k-1}}{(2k-2)!}\overline{r_n\left(\DD^{2k-1}\left(\cM\right)\right)}.
$$
For every $Q\in \AA$, we then use Theorem \ref{thm:cycleperiod} to yield 
\begin{multline*}
\CC\left(\xi_{2-2k}\left(\cM\right);Q\right) =\sum_{n=0}^{2k-2}i^{-n+1}q_{k,D,\AA}^{(n)} r_n\left(\xi_{2-2k}\left(\cM\right)\right)\\
 = \frac{\left(4\pi\right)^{2k-1}}{(2k-2)!}\sum_{n=0}^{2k-2}i^{n+1} q_{k,D,\AA}^{(n)} \overline{r_n\left(\DD^{2k-1}\left(\cM\right)\right)}\\
=-\frac{\left(4\pi\right)^{2k-1}}{(2k-2)!}\overline{\sum_{n=0}^{2k-2}i^{-n+1} q_{k,D,\AA}^{(n)} r_n\left(\DD^{2k-1}\left(\cM\right)\right)}=-\frac{\left(4\pi\right)^{2k-1}}{(2k-2)!}\overline{\CC\left(\DD^{2k-1}\left(\cM\right);Q\right)}.
\end{multline*}
This gives the claim.  
\end{proof}
Proposition \ref{prop:cycleequalnonsq} shows that \eqref{eqn:cycleequal} holds for at least one $\cM\in \Hcusp$ which maps to $g\in S_{2k}$ under $\xi_{2-2k}$ and the difference of any two such elements of $\Hcusp$ is weakly holomorphic.  To show the non-square case it hence remains to prove Corollary \ref{cor:weakly} for $D$ non-square, which we now state as a separate proposition. 
\begin{proposition}\label{prop:weakly}
If $G\in M_{2-2k}^!$ and $Q\in \QQ_D$ with $D$ non-square, then
$$
\CC\left(\DD^{2k-1}(G);Q\right)=0.
$$
\end{proposition}
Before proving Proposition \ref{prop:weakly}, we show a useful lemma.
\begin{lemma}\label{lem:qeq}
For every $k\geq 2$ and discriminant $D>0$ we have
\begin{equation}\label{eqn:qkeq}
q_{k,D,\mathcal{A}}^{(0)}=q_{k,D,\mathcal{A}}^{(2k-2)}.
\end{equation}
\end{lemma}
\begin{proof}
Denote the reduced forms $Q_j=\left[a_j,b_j,c_j\right]$ as in \eqref{eqn:reduced}.  The leading term of 
$$
Q_{k,D}\Big|_{2-2k}\left(\begin{matrix}0&1\\ 1&0\end{matrix}\right)(X) = \sum_{j=1}^r Q_j(1,-X)^{k-1}
$$
is then
\begin{equation}\label{eqn:leadingcoeff}
\sum_{j=1}^{r} c_j^{k-1}= q_{k,D,\AA}^{(2k-2)}.
\end{equation}
However, since $Q_j=Q_{j-1}\Big|_{-2}M_j=Q_{j-1}\circ M_j$ by \eqref{eqn:reduced} (where $M_j$ is defined in \eqref{eqn:Mjdef}), we have 
\begin{multline*}
Q_{k,D}\Big|_{2-2k}\left(\begin{matrix}0&1\\ 1&0\end{matrix}\right)(X)= \sum_{j=1}^{r}Q_{j-1}\circ\left(\begin{matrix}m_j & 1\\ -1 &0\end{matrix}\right)\left(\begin{matrix}1 & 0\\ 0&-1\end{matrix}\right)\left(\begin{matrix}0 & 1\\ 1&0 \end{matrix}\right)(X,1)^{k-1}\\
= \sum_{j=1}^{r}Q_{j-1}\circ\left(\begin{matrix}-1& m_j \\ 0&-1\end{matrix}\right)(X,1)^{k-1}=\sum_{j=1}^{r}Q_{j-1}\left(m_j-X,-1\right)^{k-1}\\
 =\sum_{j=1}^{r}Q_{j}\left(X-m_{j+1},1\right)^{k-1}.
\end{multline*}
Since translation does not change the leading coefficient, we conclude that the leading coefficient is 
$$
\sum_{j=1}^{r} a_{j}^{k-1}=q_{k,D,\AA}^{(0)}.
$$
Comparing with \eqref{eqn:leadingcoeff} completes the proof of the lemma.
\end{proof}
\begin{proof}[Proof of Proposition \ref{prop:weakly}]
For weakly holomorphic forms $f\in \DD^{2k-1}\left(M_{2-2k}^!\right)$, one may greatly simplify \eqref{eqn:cycleperiod}.  In this case, $r_n(f)=0$ for $0<n<2k-2$.   Moreover, it is known \cite{BGKO,KohnenZagierRational} that for every $f\in S_{2k}^!$ one has the relations
\begin{equation}\label{eqn:rrel}
i^{-n+1}r_{n}(f)=(-1)^{n+1} i^{-(2k-2-n)+1}r_{2k-2-n}(f).
\end{equation}
Hence in particular
$$
i^{1}r_{0}(f)= -i^{-(2k-2)+1}r_{2k-2}(f).
$$
Plugging this in, \eqref{eqn:cycleperiod} becomes
$$
\CC\left(f;Q\right)=i r_{0}(f)\left(q_{k,D,\AA}^{(0)}-q_{k,D,\AA}^{(2k-2)}\right).
$$
We finally use Lemma \ref{lem:qeq} to obtain $\CC\left(f;Q\right)=0$.

\end{proof}

We next consider the square case.
\begin{proposition}\label{prop:cycleequalsq}
If $D$ is a square, then \eqref{eqn:cycleequal} holds for every $\cM\in \Hcusp$ and $Q\in \QQ_{D}$.
\end{proposition}
\begin{proof}
For the equivalence classes of quadratic forms of discriminant $D$, we may take as a set of representatives (cf. Lemma 3.31 and Lemma 3.32 of \cite{Lemmermeyer} for a modern version of Gauss's theory \cite{Gauss})
$$
\left\{ \left[0,\sqrt{D},c\right]\Big| 0\leq c<\sqrt{D}\right\}.
$$
We hence have 
\begin{equation}\label{eqn:sqr}
\CC\left(\xi_{2-2k}\left(\cM\right);Q\right)=R.\int_{-\frac{c}{\sqrt{D}}}^{i\infty} \xi_{2-2k}\left(\cM\right)(z)\left(\sqrt{D}z+c\right)^{k-1}dz.
\end{equation}
The right-hand side of \eqref{eqn:sqr} was evaluated in Theorems 4.2 and 4.3 of \cite{BFK} as 
\begin{multline*}
i^{-k}D^{\frac{k-1}{2}} L_{\xi_{2-2k}\left(\cM\right)}^*\left(\zeta_{-c}^{\sqrt{D}},k\right)\\
=-\frac{\left(4\pi\right)^{2k-1}}{(2k-2)!}i^{k}D^{\frac{k-1}{2}}\overline{ L_{\DD^{2k-1}\left(\cM\right)}^*\left(\zeta_{-c}^{\sqrt{D}},k\right)}=-\frac{\left(4\pi\right)^{2k-1}}{(2k-2)!}\overline{\CC\left(\DD^{2k-1}(\cM);Q\right)},
\end{multline*}
where in the last equality we again use Theorem 4.2 of \cite{BFK} and then reverse the argument in \eqref{eqn:sqr}.  
\end{proof}

\section{Proof of Theorem \ref{thm:rmodular}}\label{sec:Shintani}
In this section, we concentrate on the Shintani lifts defined in \eqref{eqn:shindef}.  Using Theorem \ref{thm:cycleequal} to relate the Shintani lift on $S_{2k}^!$ to the classical Shintani lift on $S_{2k}$, in order to compute the Shintani lifts of all weakly holomorphic modular forms, it remains to determine the cycle integrals of the weight $2k$ Eisenstein series 
\begin{equation}\label{eqn:Eisdef}
G_{2k}(z):=-\frac{B_{2k}}{4k} + \sum_{n=1}^{\infty} \sigma_{2k-1}(n)q^n,
\end{equation}
where $B_{\ell}$ is the $\ell$th Bernoulli number and $\sigma_s(n):=\sum_{d\mid n} d^s$.  Kohnen and Zagier (see pages 240--241 of \cite{KohnenZagierRational}) related the cycle integrals of Eisenstein series to the zeta functions
$$
\zeta_{[Q]}(s):=\sum_{\substack{(m,n)\in \Gamma_{Q}\backslash\Z^2\\ Q(m,n)>0}} Q(m,n)^{-s}.
$$
Note that these series converge for $\re(s)$ sufficiently large and are independent of the representatives $Q$ of $\QQ_D/\SL_2(\Z)$.  The relation between the cycle integrals of Eisenstein series and $\zeta_{[Q]}(s)$ is given in the following lemma.
\begin{lemma}\label{lem:Eiscycle}
For every $Q\in \QQ_D$, one has 
\begin{equation}\label{eqn:Eiscycle}
\CC\left(G_{2k};Q\right)=\frac{1}{2}(-1)^k\zeta_{[Q]}\left(1-k\right).
\end{equation}
\end{lemma}
The image of $G_{2k}$ under $\shin_1$ is well-known to be a constant multiple of Cohen's \cite{Cohen} weight $k+\frac{1}{2}$ Eisenstein series (this statement is essentially contained in \cite{KohnenZagierRational}, but we add the details below for the convenience of the reader)
\begin{equation}\label{eqn:Cohendef}
\CoH_{k}(z):=\zeta(1-2k)+\sum_{\substack{m\in \N\\ (-1)^km\equiv 0,1\pmod{4}}} H\left(k,(-1)^km\right) q^{m},
\end{equation}
where for $D=D_0f^2$ with a fundamental discriminant $D_0$ and $f\in \N$ we define ($\mu$ is the usual M\"obius function)
$$
H(k,D):=L\left(1-k,\left(\frac{D_0}{\cdot}\right)\right)\sum_{d\mid f}\mu(d)\left(\frac{D_0}{d}\right)d^{k-1}\sigma_{2k-1}\left(\frac{f}{d}\right).
$$
Combining the fact that $G_{2k}$ maps to a modular form under $\shin_1$ with Theorem \ref{thm:cycleperiod},
 we are able to prove that the Shintani lift \eqref{eqn:shindef} maps weakly holomorphic modular forms to modular forms.  
\begin{proof}[Proof of Theorem \ref{thm:rmodular}]
Using Theorem \ref{thm:cycleequal} termwise immediately implies that \eqref{eqn:shinrel} holds for every $\cM\in \Hcusp$, yielding (2).

Since $\xi_{2-2k}$ is surjective \cite{BruinierFunke} and $g^c$ is a cusp form if $g$ is a cusp form, (1) then follows from (2) and the classical results of Shintani \cite{Shintani} for cusp forms.  

In order to prove (3), we use the following identity of Siegel \cite{SiegelEisen}:
$$
\sum_{Q\in \QQ_D/\SL_2(\Z)}\zeta_{[Q]}(1-k) = \zeta(1-k)H(k,D).
$$
Combining this with Lemma \ref{lem:Eiscycle} and noting that $\ChiQ_{1}(Q)=1$ for every $Q$ yields the third statement of the theorem.
\end{proof}

\end{document}